\documentclass[12pt]{amsart}

\usepackage{tikz}
\usepackage{graphicx}
\usepackage{amsmath}
\usepackage{amssymb}
\usepackage{amsfonts}
\usepackage{verbatim}
\usepackage{scalefnt}
\usepackage{multirow}
\usepackage{enumerate}
\usepackage{mathtools}
\usepackage{float}
\usepackage{enumitem}
\usepackage{tabularx}
\restylefloat{figure}
\restylefloat{table}
\usepackage[margin=3cm]{geometry}
\usepackage{xcolor}

\usepackage{fancyhdr}
\usepackage{hyperref}

\newcommand{\F}{{\mathbb F}}

\newtheorem{theorem}{Theorem}

\newtheorem{corollary}[theorem]{Corollary}

\newtheorem*{definition*}{Definition}

\newtheorem{claim}{Claim}[theorem]
\renewcommand{\theclaim}{\arabic{claim}}
\newenvironment{proofofclaim}[1][\proofname\ of Claim \theclaim]{%
  \proof[#1]%
  
}{\endproof}

\numberwithin{equation}{section}
\numberwithin{theorem}{section}

\newcommand{\genlegendre}[4]{%
  \genfrac{(}{)}{}{#1}{#3}{#4}%
  \if\relax\detokenize{#2}\relax\else_{\!#2}\fi
}
\newcommand{\legendre}[3][]{\genlegendre{}{#1}{#2}{#3}}

\usepackage[backend=biber]{biblatex} %biblatex mit biber laden
\DeclareFieldFormat{pages}{#1}
\title[Maximal line-free sets in $\mathbb{F}_p^n$]%
  {Maximal line-free sets in $\mathbb{F}_p^n$}

\author[Elsholtz et al.]{Christian Elsholtz}
\address{Christian Elsholtz and Jakob F\"uhrer, Institute of Analysis and Number Theory, Graz University of Technology,
Kopernikusgasse 24/II,
8010 Graz, Austria.}
\email{elsholtz@math.tugraz.at}

\author[]{Jakob F\"uhrer}

\email{jakob.fuehrer@tugraz.at}

\author[]{Erik F\"uredi}
\address{Erik F\"uredi, ELTE E\"otv\"os Lor\'and University Faculty of Science, 1117 Budapest, P\'azm\'any P\'eter s\'et\'any 1/A, Hungary.}
 \email{erikfuredi@gmail.com}

\author[]{Benedek Kov\'acs}
\address{Benedek Kov\'acs, ELTE Linear Hypergraphs Research Group, Eötvös Loránd University, 1117 Budapest, P\'azm\'any P\'eter s\'et\'any 1/A, Hungary.}
 \email{benoke98@student.elte.hu}

\author[]{P\'eter P\'al Pach}
\address{P\'eter P\'al Pach, Department of Computer Science and Information Theory,  Budapest University of Technology and Economics,  M\H{u}egyetem rkp. 3., H-1111 Budapest, Hungary;  MTA-BME Lend\"ulet Arithmetic Combinatorics Research Group,
  M\H{u}egyetem rkp. 3., H-1111 Budapest, Hungary}
 \email{pach.peter@vik.bme.hu}

\author[]{D\'aniel G\'abor Simon}
\address{D\'aniel G\'abor Simon, HUN-REN Alfr\'ed R\'enyi Institute of Mathematics, Re\'altanoda street 13-15, H-1053 Budapest, Hungary.;  ELTE E\"otv\"os Lor\'and University Faculty of Science, 1117 Budapest, P\'azm\'any P\'eter s\'et\'any 1/A, Hungary.}
 \email{dgs45@cantab.ac.uk}

\author[]{N\'ora Velich}
\address{N\'ora Velich, University of Cambridge, Lucy Cavendish College, Lady Margaret Road, Cambridge CB3 0BU, UK.}
 \email{nzv20@cam.ac.uk}

\date{}

\thanks{}

\begin{document}

\begin{abstract}

We study subsets of $\mathbb{F}_p^n$ that do not contain progressions of length $k$. We denote by $r_k(\mathbb{F}_p^n)$ the cardinality of such subsets containing a maximal number of elements.

In this paper we focus on the case $k=p$ and therefore sets containing no full line.
A~trivial lower bound $r_p(\mathbb{F}_p^n)\geq(p-1)^n$ is achieved by a hypercube of side length $p-1$ and it is known that equality holds for $n\in\{1,2\}$. We will however show that $r_p(\mathbb{F}_p^3)\geq (p-1)^3+p-2\sqrt{p}$, which is the first improvement in the three dimensional case that is increasing in $p$. 

We will also give the upper bound $r_p(\mathbb{F}_p^{3})\leq p^3-2p^2-(\sqrt{2}-1)p+2$ as well as generalizations for higher dimensions.

Finally we present some bounds for individual $p$ and $n$, in particular $r_5(\mathbb{F}_5^{3})\geq 70$ and $r_7(\mathbb{F}_7^{3})\geq 225$ which can be used to give the asymptotic lower bound $4.121^n$ for $r_5(\mathbb{F}_5^{n})$ and $6.082^n$ for $r_7(\mathbb{F}_7^{n})$.

\end{abstract}

\date{\today}
\maketitle

\section{Introduction}

In the intersection of finite geometry and extremal combinatorics numerous problems about finding maximal subsets of affine or projective spaces avoiding certain configurations have been studied. One natural question asks for bounds on the cardinality of subsets of the $n$-dimensional affine space over a finite field $\mathbb{F}_q$ that do not contain a full line.

We denote by $r_k(\mathbb{F}_p^n)$ the cardinality of a subset $S\subseteq \mathbb{F}_p^n$, containing a maximal number of elements such that $S$ contains no $k$ points in arithmetic progression. Note that in the case when $k=p$ is a prime, $k$-progressions in  $\mathbb{F}_p^n$ correspond to lines in the $n$-dimensional affine space and we are therefore interested in bounds on $r_p(\mathbb{F}_p^n)$.

When $p=3$ the problem coincides with the cap set problem, a well studied area where one can use the fact that  $x,y,z$ form a line exactly when they fulfil a non-trivial linear equation $ax+by+cz=0$ where $a+b+c=0$. Namely, with $a=b=c=1$. Ellenberg and Gijswijt \cite{EG} gave the first exponential improvement to the trivial upper bound of $3^n$ with $r_3(\mathbb{F}_3^{n})<2.756^n$, for large enough $n$ which was further improved by Jiang \cite{Jiang} by a factor of $\sqrt{n}$. The best lower bound was given by Romera-Paredes et al.~\cite{DeepMind} with $r_3(\mathbb{F}_3^{n})>2.220^n$, for large enough $n$. The exact values of $r_3(\mathbb{F}_3^{n})$ are known up to $n=6$, where $r_3(\mathbb{F}_3^{6})=112$ was proven by Potechin \cite{Potechin}.

For the general case surprisingly few results on $r_p(\F_p^n)$ are known. There is the trivial lower bound $r_p(\mathbb{F}_p^n)\geq(p-1)^n$ achieved by a hypercube of side length $p-1$. Jamison \cite{Jamison} and Brouwer and Schrijver~\cite{BS} independently proved that this is sharp for $n=2$. For $n=3$ the only improvement to this construction was by a single point described in the post of Zare in a mathoverflow thread \cite{Mathoverflow}. We will prove the following lower bounds:

\begin{theorem}
\label{LBhighp}
Let $p\geq 5$ be a prime, then $$r_p(\mathbb{F}_p^3) \geq (p-1)^3+p-2\sqrt{p}=p^3-3p^2+4p-2\sqrt{p}-1.$$
\end{theorem}

This can be improved in some special cases.

\begin{theorem}
\label{LB7}
Let $p$ be a prime with $p\equiv 7$ \textnormal{(mod $24$)}, then 
$$r_p(\mathbb{F}_p^3)\geq(p-1)^3+(p-1)=p^3-3p^2+4p-2.$$

Moreover, $r_7(\mathbb{F}_7^3)\geq 225$.
\end{theorem}

The simple upper bound $r_p(\mathbb{F}_p^n)\leq p^n-\frac{p^n-1}{p-1}$ was given by Aleksanyan and Papikian \cite{AP} and is achieved by removing at least one point from each line going through a fixed point. In particular $r_p(\mathbb{F}_p^3)\leq p^3-p^2-p-1$. The stronger bounds $r_p(\mathbb{F}_p^n)\leq p^n-2p^{n-1}+1$ and $r_p(\mathbb{F}_p^3)\leq p^3-2p^{2}+1$ can be obtained by a result of Sziklai~\cite[Proposition~4.1]{Sziklai} (see also \cite{ball2009},~\cite{BDGP}). We will give the following new bounds:

\begin{theorem}
\label{ubmain}
Let $p\geq 3$ be a prime, $k\in \{3,4,\dots,p\}$ and $n\in \mathbb{N}$, then

$${r_k(\mathbb{F}_p^{n+1})}\leq \frac{2(p^
{n+1}-1)r_k(\mathbb{F}_p^n)+p^n-\sqrt{4(p^
{n+1}-1)r_k(\mathbb{F}_p^n)(p^n-r_k(\mathbb{F}_p^n))+p^{2n}}}{2p^n},$$
\end{theorem}

where the three-dimensional case gives the following corollary.

\begin{corollary}\label{cor-upp}
Let $p\geq 3$ be a prime, then $$r_p(\mathbb{F}_p^3) \leq\frac{2p^5-4p^4+2p^3-p^2+4p-2-\sqrt{8p^6-20p^5+17p^4-12p^3+20p^2-16p+4}}{2p^2},$$
in particular,
$$r_p(\mathbb{F}_p^3)\leq p^3-2p^2-(\sqrt{2}-1)p+2.$$

\end{corollary}

For other dimensions, there is the lower bound $r_p(\mathbb{F}_p^{2p})\geq p(p-1)^{2p-1}$ due to Frankl et al.~\cite{FGR}, using large sunflower-free sets.

We found a $70$ point $5$-progression-free set in $\mathbb{F}_5^3$ via a branch and cut approach (see Figure \ref{fig70}) and we will show the following upper bounds for small primes.

\begin{theorem}\label{thm-r5}
$r_5(\mathbb{F}_5^3)<74.$
\end{theorem}

\begin{theorem}\label{thm-r7}
$r_7(\mathbb{F}_7^3)<243. $
\end{theorem}

One can use the product $S_1\times S_2$ of two line-free sets $S_1\in \mathbb{F}_p^{n_1}$, $S_2\in \mathbb{F}_p^{n_2}$ to get a line-free set in the higher dimension $n_1+n_2$. This construction also provides us the lower bound $|S_1|^{1/{n_1}}$ for $\alpha_p:=\lim\limits_{n \rightarrow \infty}(r_p(\mathbb{F}_p^n))^{1/n}$ and therefore the asymptotic lower bound $(|S_1|^{1/{n_1}}-o(1))^n$ for $r_p(\mathbb{F}_p^n)$ (see e.g. \cite{SET}, \cite{Pach22}).
The strongest known lower bound for general $p$ is  $\alpha_p\geq p^{1/{2p}}(p-1)^{(2p-1)/{2p}}$ using the results of Frankl et al.~\cite{FGR}, however for small primes the new three-dimensional lower bounds $r_5(\mathbb{F}_5^3)\geq 70$ and  $r_7(\mathbb{F}_7^3)\geq 225$ give better lower bounds, namely, $\alpha_5\geq 4.121$  and $\alpha_7\geq 6.082$.

We will also show the following explicit lower bound for arbitrary dimension (see Table \ref{table:1} for comparisons).

\begin{theorem}
\label{LBhighn}
Let $p\geq 3$ be a prime, then 
$r_p(\mathbb{F}_p^n)\geq(p-1)^n+\frac{n-2}{2}(p-1)(p-2)^{n-3}.$
\end{theorem}

\section{Related results}

\begin{itemize}
    \item Davis and Maclagan \cite{SET} studied the card game SET, where the cards can be described as points in $\mathbb{F}_3^4$ and one is interested whether the displayed cards form a cap set. The best lower bound for cap sets prior to the work of Romera-Paredes et al.~\cite{DeepMind} was due to Tyrell \cite{Tyrrell}, building on the construction of Edel \cite{Edel}. Naslund~\cite{naslund} announced the improvement $\alpha_3\geq 2.2208$. Elsholtz and Lipnik \cite{ElsholtzLipnik} and Elsholtz and Pach \cite{ElsholtzPach} studied cap sets in other spaces than $\mathbb{F}_3^n$.

    \item Croot et al.~\cite{CrootLevPach} gave an upper bound for $3$-progression-free sets in $\mathbb{Z}_4^n$ that is exponentially smaller than $4^n$. Their method also led to the result of Ellenberg and Gijswijt \cite{EG}. Petrov and Pohoata \cite{PetrovPohoata} gave an improved upper bound for $3$-progression-free sets in $\mathbb{Z}_8^n$, Pach and Palincza \cite{PachPalincza} gave both upper and lower bounds for $6$-progression-free sets in $\mathbb{Z}_6^n$. Elsholtz et al.~\cite{EKL} studied the general case of $k$-progression-free sets in $\mathbb{Z}_m^n$. An overview on known bounds is given by Pach \cite{Pach22}.

    \item Moser \cite{MoserProblem} asked for the maximal size of a subset of $\{1,2,\dots,k\}^n$ without a geometric line. Similarly, Hales and Jewett asked for a subset without a combinatorial line. The result of Furstenberg and Katznelson \cite{FK} also known as the density Hales–Jewett theorem implies that in both cases these sets have to be asymptotically smaller than $k^n$ as $n$ tends to infinity. Polymath \cite{PolymathMoser} gave some explicit bounds for special cases.

    \item Sets that intersect every affine subspace of codimension $s$ are called $s$-blocking sets. The complement of a line-free set in a finite $n$-dimensional affine space is therefore also called an $(n-1)$-blocking set. It is known that the union of any $n$ independent lines intersecting in a single point form a $1$-blocking set in $\mathbb{F}_p^n$ which is optimal (see e.g. \cite{AlonTFHA}, \cite{BS}, \cite{Jamison}). However, for $(n-1)$-blocking sets, the union of $n$ independent hyperplanes, which seems to be the obvious algebraic construction, are not optimal, as will be shown in this paper.  Bishnoi et al.~\cite{BDGP} gave several upper bounds for the size of $s$-blocking sets.

\end{itemize}

\section{Notation}
We write $\mathbb{Z}_n$ for $\mathbb{Z}/n\mathbb{Z}$ and $\mathbb{F}_p=\mathbb{Z}_p$ is the field with $p$ elements whenever $p$ is a prime.

We write $[k,\ell]$ for the set $\{k,k+1,\dots,\ell\}$ either as a subset of $\mathbb{Z}$ or of $\mathbb{F}_p$.

We use both row and column vectors for the elements of $\mathbb{F}_p^n$ and we call these elements points.

Given a subset $S\subseteq \mathbb{F}_p^3$ we call the image of $S\cap (\{j\}\times\mathbb{F}_p^2)$ under the projection $\phi\colon\mathbb{F}_p^3\longrightarrow\mathbb{F}_p^2$, $(a,b,c)\mapsto\left(b,c\right)$ the $j$-layer of $S$.

\section{Proofs  of the upper bounds}

\begin{proof}[Proof of Theorem \ref{ubmain}]
Let $A\subseteq \mathbb{F}_p^{n+1}$ be $k$-progression-free with $|A|=r_k(\mathbb{F}_p^{n+1})$. We count the number of the point pairs on every $n$-dimensional affine hyperplane $$s = \left| \{ (\{a,b\},S)\ |\ a,b\in A,\ a\ne b,\ a,b\in S,\  \text{$S$ is an $n$-dimensional affine hyperplane}\}\right| .$$ On every hyperplane, the number of points is at most $r_k(\mathbb{F}_p^n)$. Firstly, we assume $r_k(\mathbb{F}_p^{n+1})\geq(p-1)r_k(\mathbb{F}_p^n)$, then the sum of number of point pairs for $p$ parallel hyperplanes is maximal if there are $p-1$ hyperplanes with $r_k(\mathbb{F}_p^n)$ points and one with $r_k(\mathbb{F}_p^{n+1})-(p-1)r_k(\mathbb{F}_p^n)$. 

There are $\frac{p^{n+1}-1}{p-1}$ disjoint sets of parallel hyperplanes, so
\begin{equation*} s\leq\Big(\frac{p^{n+1}-1}{p-1}\Big)\Big((p-1)\binom{r_k(\mathbb{F}_p^n)}{2}+
\binom{r_k(\mathbb{F}_p^{n+1})-(p-1)r_k(\mathbb{F}_p^{n})}{2}\Big).
\end{equation*}
Note here that this inequality still holds if $r_k(\mathbb{F}_p^{n+1})<(p-1)r_k(\mathbb{F}_p^n)$, as in this case the number of point pairs is clearly less than $$(p-1)\binom{r_k(\mathbb{F}_p^n)}{2}$$ and $$\binom{r_k(\mathbb{F}_p^{n+1})-(p-1)r_k(\mathbb{F}_p^{n})}{2}\geq 0.$$

On the other hand, every point pair defines a line that is included in exactly $\frac{p^n-1}{p-1}$ $n$-dimensional affine hyperplanes, so $$s=\frac{p^n-1}{p-1}\binom{r_k(\mathbb{F}_p^{n+1})}{2}.$$

We get the quadratic inequality $$p^n\big(r_k(\mathbb{F}_p^{n+1})\big)^2-\big(p^n+2(p^{n+1}-1)r_k(\mathbb{F}_p^{n})\big)r_k(\mathbb{F}_p^{n+1})+(p^{n+2}-p)\big(r_k(\mathbb{F}_p^{n})\big)^2\geq 0$$ with roots $$\frac{2(p^
{n+1}-1)r_k(\mathbb{F}_p^n)+p^n\pm\sqrt{4(p^
{n+1}-1)r_k(\mathbb{F}_p^n)(p^n-r_k(\mathbb{F}_p^n))+p^{2n}}}{2p^n}.$$

As $${r_k(\mathbb{F}_p^{n+1})}\leq p(r_k(\mathbb{F}_p^n))$$ but 
\begin{equation*}
\begin{split}   
 & \frac{2(p^{n+1}-1)r_k(\mathbb{F}_p^n)+p^n+\sqrt{4(p^
{n+1}-1)r_k(\mathbb{F}_p^n)(p^n-r_k(\mathbb{F}_p^n))+p^{2n}}}{2p^n} \\
 > \; &  p (r_k(\mathbb{F}_p^n))+\frac{1}{2}-\frac{r_k(\mathbb{F}_p^n)}{p^n}+\frac{\sqrt{p^{2n}}}{2p^n}\geq p (r_k(\mathbb{F}_p^n))+\frac{1}{2}-1+\frac{1}{2}= p (r_k(\mathbb{F}_p^n)),
\end{split}
\end{equation*}

the theorem follows.
\end{proof}

\begin{proof}[Proof of Corollary \ref{cor-upp}]
The first statement follows immediately from Theorem \ref{ubmain} using $r_p(\mathbb{F}_p^2)=(p-1)^2$. For the second statement we are using that $8p^6-20p^5+17p^4-12p^3+20p^2-16p+4$ can be bounded by $(2\sqrt{2}p^3-5/\sqrt{2}p^2)^2$ from below for $p\geq 3$ and we get 
\begin{equation*}
\begin{split}   
r_p(\mathbb{F}_p^3)\leq & \; p^3-2p^2+p-\frac{1}{2}+\frac{2}{p}-\frac{1}{p^2}-\sqrt{2}p+\frac{5}{2 \sqrt{2}} \\ \leq & \;  p^3-2p^2-(\sqrt{2}-1)p-\frac{1}{2}+\frac{2}{3}+\frac{5}{2 \sqrt{2}} \\ \leq  & \; p^3-2p^2-(\sqrt{2}-1)p+2.
\end{split}
\end{equation*}
\end{proof}

\begin{proof}[Proof of Theorem \ref{thm-r5}]
Assume that $S\subseteq \mathbb{F}_5^3$ is a $5$-progression-free set of size $74$. We will compute a weighted sum over all lines containing $4$ points to reach a contradiction.

Let us call a line containing exactly $r$ points an $r$-line. Let $l$ be a $4$-line in $S$ and let $H_1,H_2,\dots,H_6$ be the planes containing $l$. Then $\sum_{i=1}^6|H_i\cap S|=(74-4)+6\cdot 4=94$. Note that $r_5(\mathbb{F}_5^2)=16$ and $r_4(\mathbb{F}_5^2)=11$, which can be easily checked by computer search. Therefore, $|H_i\cap A|\geq 94-5\cdot 16=14$ for all $i$ and there is no plane in $\mathbb{F}_5^3$ containing $12$ or $13$ points.
Hence, there are five different distributions for the number of points in five parallel planes:
\begin{itemize}
    \item [(a)]$\{10,16,16,16,16\}$
    \item [(b)]$\{11,15,16,16,16\}$
    \item [(c)]$\{14,14,14,16,16\}$
    \item [(d)]$\{14,14,15,15,16\}$
    \item [(e)]$\{14,15,15,15,15\}.$
\end{itemize}
Denote by $a,b,c,d,e$ the number of classes of parallel planes having these distributions. Note that 
\begin{equation}\label{eq-sum}
a+b+c+d+e=31.
\end{equation}

If we compare the number of pairs of points in each plane with the total number of pairs we get $(\binom{10}{2}+4\binom{16}{2})a+(\binom{11}{2}+\binom{15}{2}+3\binom{16}{2})b+(3\binom{14}{2}+2\binom{16}{2})c+(2\binom{14}{2}+2\binom{15}{2}+\binom{16}{2})d+(\binom{14}{2}+4\binom{15}{2})e=6\binom{74}{2}$

\begin{equation}\label{eq-pairs}
\Leftrightarrow 525a+520b+513c+512d+511e=16206,
\end{equation}
since each pair lies in exactly six planes.

Now denote by $A$, $B$, and $C$ the number of pairs $(\ell,H)$ where $H$ is a hyperplane containing $16$, $15$ and $14$ points, respectively and $\ell\subseteq H$ is a $4$-line. Again let $\ell$ be a $4$-line and let $H_1,H_2,\dots,H_6$ be the planes containing $\ell$. Then $$\{\left| H_i\cap S\right| |\;i\in[1,6]\}\in\{\{14,16,16,16,16,16\},\{15,15,16,16,16,16\}\}$$ as multisets and therefore
\begin{equation}\label{eq-main}
A-2B-5C=0
\end{equation}
To bound the size of $A$, $B$ and $C$ we need the following claims.

\begin{claim}
Every plane containing $16$ points contains at least twelve $4$-lines.
\end{claim}
\begin{proofofclaim}
Consider a plane $H$ containing $16$ points and let $x_i$ be the number of $i$-lines in $H$ for $i\in\{1,2,3,4\}$. By double counting the points in $H$ we get $x_1+2x_2+3x_3+4x_4=6\cdot16=96$ and by double counting the pairs of points in $H$ we get $x_2+3x_3+6x_4=\binom{16}{2}=120$. By taking the difference of the two equations we get $-x_1-x_2+2x_4=24$, implying that $2x_4\geq 24$.
\end{proofofclaim}

\begin{claim}
For $m\in \{14,15\}$, every plane containing $m$ points contains at most $m$ $4$-lines.
\end{claim}
\begin{proofofclaim}
As $5\cdot3+1=16> m$, every point in $S$ can be contained in at most four $4$-lines and therefore the number of $4$-lines in the plane is bounded from above by $\frac{4m}{4}=m.$
\end{proofofclaim}

Finally combining \eqref{eq-sum}, \eqref{eq-pairs} and \eqref{eq-main} we obtain the following system of linear equations and inequalities.

\begin{equation*}
\begin{split}   
&a+b+c+d+e=31 \\
&525a+520b+513c+512d+511e=16206 \\
&A-2B-5C=0 \\
&A\geq 48a+36b+24c+12d \\
&B\leq 15b+30d+60e \\
&C\leq 42c+28d+14e \\
&a,b,c,d,e,A,B,C\geq 0,
\end{split}
\end{equation*}
which does not have any integral solution, a contradiction to $|S|=74$.

\end{proof}

\begin{proof}[Proof of Theorem \ref{thm-r7}]
Assume that $S\subseteq \mathbb{F}_7^3$ is a $7$-progression-free set of size $243$. Note that we have the following bounds.

\begin{claim}
Every plane containing $36$ points contains at least $18$ $6$-lines and every plane containing $35$, $34$ or $33$ contains at most $33$, $30$, $28$ $6$-lines, respectively.
Moreover, $r_7(\mathbb{F}_7^2)=36$ and $r_6(\mathbb{F}_7^2)=29$.
\end{claim}
\begin{proofofclaim}
Consider a plane $H$ containing $m$ points and let $x_i$ be the number of $i$-lines in $H$ for $i\in [0,6]$. There are $56$ lines in the plane and therefore 
\begin{equation}\label{eq-sublin}
x_0+x_1+x_2+x_3+x_4+x_5+x_6=56
\end{equation}
By double counting the points in $H$ we get 
\begin{equation}\label{eq-lin}
x_1+2x_2+3x_3+4x_4+5x_5+6x_6=8m
\end{equation}

and by double counting the pairs of points in $H$ we get 
\begin{equation}\label{eq-qua}
x_2+3x_3+6x_4+10x_5+15x_6=\binom{m}{2}.
\end{equation}
If $m=36$ 
we take the difference of the \eqref{eq-qua} and two times \eqref{eq-lin} and get $-2x_1-3x_2-3x_3-2x_4+3x_6=54$, implying that $x_6\geq 18$.
If $m\in\{33,34,35\}$, then by taking three times \eqref{eq-sublin} minus two times \eqref{eq-lin} plus \eqref{eq-qua} we get $3x_0+x_1+x_4+3x_5+6x_6=168-16m+\binom{m}{2}$ and therefore $6x_6\leq168-16m+\binom{m}{2}$ which gives the desired bounds.

The last two claims can be easily checked by computer search.
\end{proofofclaim}

If we now proceed analogously to the proof of Theorem \ref{thm-r5} we again arrive at a contradiction.
\end{proof}

\section{Proofs of the lower bounds}

\begin{proof}[Proof of Theorem \ref{LBhighn}]
We consider three different types of $2$-dimensional layers: 
\begin{itemize}
    \item $A:=[0,p-2]^{2}$,
    \item $B:=[0,p-1]^2\setminus \{(i,i)\ |\ i \in [0,p-1]\}\setminus \big(\{p-1\}\times [0,\frac{p-3}{2}]\big)\setminus  \big([0,\frac{p-3}{2}]\times \{p-1\}\big)$,
    \item $C:= \{(i,i)\ |\ i \in [0,\frac{p-3}{2}]\}$,
\end{itemize}

and three disjoint subsets of $\mathbb{F}_p^{n-2}$:
\begin{itemize}
    \item $\mathcal{A}:=[0,p-3]^{n-2}$,
    \item $\mathcal{B}:=[0,p-2]^{n-2}\setminus [0,p-3]^{n-2}$,
    \item $\mathcal{C}:=\bigcup_{j\in[1,n-2]}\{x\in\mathbb{F}_p^{n-2} \; | \; (x_j=p-1) \land (x_i\in[0,p-3] \; \forall i\neq j)\}.$
\end{itemize}
We show that $S:=(\mathcal{A}\times A)\cup (\mathcal{B}\times B)\cup (\mathcal{C}\times C) $ is $p$-progression-free.

First consider the case $n=3$.
Let  $L:=\{(a_1,a_2,a_3)+(b_1,b_2,b_3)i \ |\  i\in [0,p-1]\}$ be a $p$-progression in $\mathbb{F}_p^3$ with $a_1,a_2,a_3,b_1,b_2,b_3\in\mathbb{F}_p$.
\begin{itemize}

\item Case 1: $b_1=0$ and $a_1\neq p-2:$

\noindent
$[0,p-2]^2$ is $p$-progression-free and $|\{(i,i)\ |\ i \in [0,\frac{p-3}{2}]\}|<p$, therefore $L$ is not contained in $S$.
\smallskip
\item Case 2: $b_1=0$ and $a_1= p-2:$

\noindent
$L':=\{(a_2,a_3)+(b_2,b_3)i\ |\ i\in [0,p-1]\}$ and $\{(i,i)\ |\ i \in [0,p-1]\}$ are both lines in $\mathbb{F}_p^2$. If they are not parallel or they are equal, they do intersect, and $L$ is not contained in $S$. Otherwise we can rewrite $L'=\{(i,c+i)\ |\ i\in [0,p-1]\}$ with $c\in [1,p-1].$ If $c\in [1,\frac{p-1}{2}]$ then  $c+(p-1)\in[0,\frac{p-3}{2}]$ (choose $i=p-1$) and $(p-2,p-1,c+(p-1))\in L\setminus S.$

Similarly if $c\in [\frac{p+1}{2},p-1]$, then $p-1-c\in[0,\frac{p-3}{2}]$ (choose $i=p-1-c$) and $(p-2,p-1-c,p-1)\in L\setminus S.$ Therefore, $L$ is not contained in $S$.
\smallskip
\item Case 3: $b_1\neq 0$:

\noindent
Without the loss of generality let $b_1=1$ and $a_1=p-2$. If $b_2=b_3=0$ then $L$ is not contained in $S$ because  the $(p-2)$-layer and $(p-1)$-layer of $S$ have no common point.
Otherwise, without the loss of generality, let $b_2\neq 0$ and therefore $\{a_2+b_2i\ |\ i\in [0,p-1]\}=[0,p-1].$ Assume that $L\subseteq S$. Then $a_2=p-1$ and $a_3\in [\frac{p-1}{2},p-2]$ because the $(p-2)$-layer is the only layer containing points with the coordinate $p-1$. Since the  $(p-1)$-layer does not have coordinates in $[\frac{p-1}{2},p-2]$, also $b_3\neq 0$ and consequently  $\{a_3+b_3i\ |\ i\in [0,p-1]\}=[0,p-1].$ As before it follows that $a_3=p-1$ contradicting that $L\subseteq S$. Thus, $L$ is not contained in $S$ and $S$ is $p$-progression-free.
\end{itemize}

Now consider $n>3$. We have already seen that every layer is $p$-progression-free, so we only consider progressions  $L:=\{a+bi \ |\  i\in [0,p-1]\}$ visiting $p$ non-empty layers. Let $m$ be the number of non-zero entries in the first $n-2$ coordinates of $b$. Since only layers of type $C$ are placed where one of the first $n-2$ coordinates is $p-1$ and all layers where two of the first $n-2$ coordinates are $p-1$ are empty, $m$ is also the number of type $C$ layers visited by $L$ and $m\leq p$.

\begin{itemize}
    \item If $m=1$, $L$ is not contained in $S$, analogously to the $3$-dimensional case.

    \item If $m\geq 2$ the last two coordinates of every point in $L$ are equal, since the projection of $L$ in the last two coordinates is a line containing two points in the main diagonal, or it is a single point in the main diagonal. Now since only layers of type $B$ are placed where one of the first $n-2$ coordinates is $p-2$, $L$ also visits a layer of type $B$. Therefore $L$ is not contained in $B$ because layers of type $B$ contain no points on the main diagonal.
\end{itemize}

Finally, note that layers of type $A$ and $B$ contain $(p-1)^2$ points and layers of type $C$ contain $(p-1)/2$ points and thus $$|S|=(p-1)^2(p-1)^{n-2}+\frac{p-1}{2}(n-2)(p-2)^{n-3}=(p-1)^n+\frac{n-2}{2}(p-1)(p-2)^{n-3}.$$

\end{proof}

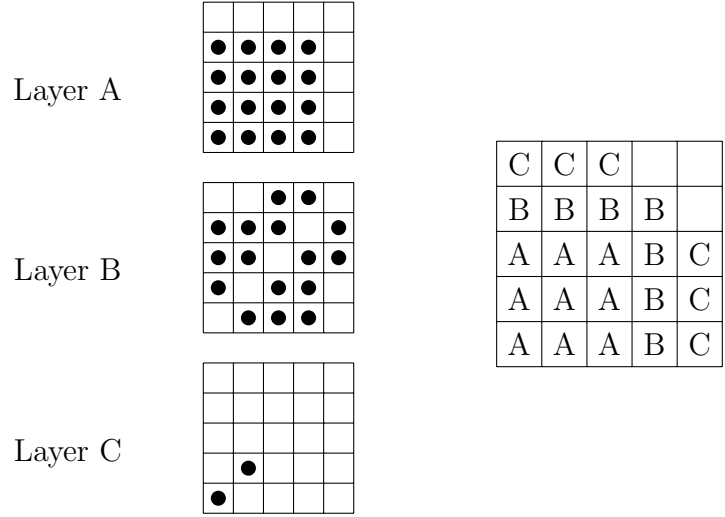
\begin{figure}
\centering
\begin{tikzpicture}[yscale=1,scale=0.5]
\begin{scope}[scale=0.8]
\node (A) at (-5,1.5) {Layer A};
\node (B) at (-5,-4.5) {Layer B};
\node (C) at (-5,-10.5) {Layer C};
\begin{scope}
\draw[step=1.0,black,thin,xshift=-0.5cm,yshift=-0.5cm] (0,0) grid (5,5);

\fill (0,0) circle (0.25);
\fill (0,1) circle (0.25);
\fill (0,2) circle (0.25);
\fill (0,3) circle (0.25);
\fill (1,0) circle (0.25);
\fill (1,1) circle (0.25);
\fill (1,2) circle (0.25);
\fill (1,3) circle (0.25);
\fill (2,0) circle (0.25);
\fill (2,1) circle (0.25);
\fill (2,2) circle (0.25);
\fill (2,3) circle (0.25);
\fill (3,0) circle (0.25);
\fill (3,1) circle (0.25);
\fill (3,2) circle (0.25);
\fill (3,3) circle (0.25);

\end{scope}

\begin{scope}[yshift=-6cm, xshift=0cm]

\draw[step=1.0,black,thin,xshift=-0.5cm,yshift=-0.5cm] (0,0) grid (5,5);
\fill (0,1) circle (0.25);
\fill (0,2) circle (0.25);
\fill (0,3) circle (0.25);
\fill (1,0) circle (0.25);
\fill (1,2) circle (0.25);
\fill (4,2) circle (0.25);
\fill (2,0) circle (0.25);
\fill (2,1) circle (0.25);
\fill (2,3) circle (0.25);
\fill (3,0) circle (0.25);
\fill (2,4) circle (0.25);
\fill (3,2) circle (0.25);
\fill (3,4) circle (0.25);
\fill (3,1) circle (0.25);
\fill (4,3) circle (0.25);
\fill (1,3) circle (0.25);

\end{scope}
\begin{scope}[yshift=-12cm, xshift=0cm]

\draw[step=1.0,black,thin,xshift=-0.5cm,yshift=-0.5cm] (0,0) grid (5,5);
\fill (0,0) circle (0.25);
\fill (1,1) circle (0.25);

\end{scope}
\end{scope}

\begin{scope}[xshift=8cm,yshift=-5.5cm,scale=1.2]
\draw[step=1.0,black,thin,xshift=-0.5cm,yshift=-0.5cm] (0,0) grid (5,5);

\node () at (0,0) {A};
\node () at (0,1) {A};
\node () at (0,2) {A};
\node () at (1,0) {A};
\node () at (1,1) {A};
\node () at (1,2) {A};
\node () at (2,0) {A};
\node () at (2,1) {A};
\node () at (2,2) {A};
\node () at (0,3) {B};
\node () at (1,3) {B};
\node () at (2,3) {B};
\node () at (3,3) {B};
\node () at (3,2) {B};
\node () at (3,1) {B};
\node () at (3,0) {B};
\node () at (4,0) {C};
\node () at (4,1) {C};
\node () at (4,2) {C};
\node () at (0,4) {C};
\node () at (1,4) {C};
\node () at (2,4) {C};

\end{scope}

\end{tikzpicture}
\caption{A description of the line-free set in Theorem~\ref{LBhighn} for $p=5$ and $n=4$.}
\end{figure}

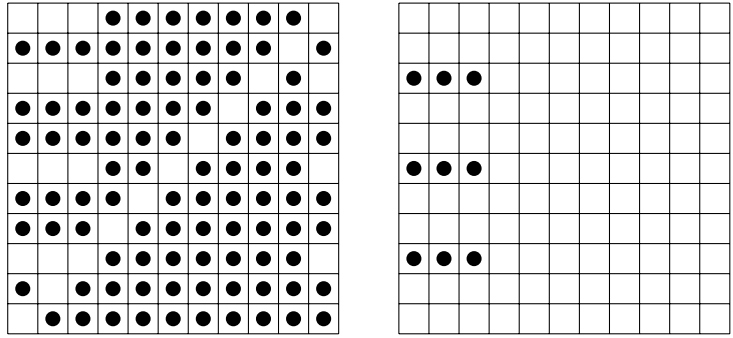
\begin{figure}
\centering
\begin{tikzpicture}[yscale=1,scale=0.4]

\begin{scope}
\draw[step=1.0,black,thin,xshift=-0.5cm,yshift=-0.5cm] (0,0) grid (11,11);

\fill (0,1) circle (0.25);

\fill (0,3) circle (0.25);
\fill (0,4) circle (0.25);

\fill (0,6) circle (0.25);
\fill (0,7) circle (0.25);

\fill (0,9) circle (0.25);

\fill (1,0) circle (0.25);

\fill (1,3) circle (0.25);
\fill (1,4) circle (0.25);

\fill (1,6) circle (0.25);
\fill (1,7) circle (0.25);

\fill (1,9) circle (0.25);

\fill (2,0) circle (0.25);
\fill (2,1) circle (0.25);
\fill (2,3) circle (0.25);
\fill (2,4) circle (0.25);

\fill (2,6) circle (0.25);
\fill (2,7) circle (0.25);

\fill (2,9) circle (0.25);

\fill (3,0) circle (0.25);
\fill (3,1) circle (0.25);
\fill (3,2) circle (0.25);
\fill (3,4) circle (0.25);
\fill (3,5) circle (0.25);
\fill (3,6) circle (0.25);
\fill (3,7) circle (0.25);
\fill (3,8) circle (0.25);
\fill (3,9) circle (0.25);
\fill (3,10) circle (0.25);
\fill (4,0) circle (0.25);
\fill (4,1) circle (0.25);
\fill (4,2) circle (0.25);
\fill (4,3) circle (0.25);
\fill (4,5) circle (0.25);
\fill (4,6) circle (0.25);
\fill (4,7) circle (0.25);
\fill (4,8) circle (0.25);
\fill (4,9) circle (0.25);
\fill (4,10) circle (0.25);
\fill (5,0) circle (0.25);
\fill (5,1) circle (0.25);
\fill (5,2) circle (0.25);
\fill (5,3) circle (0.25);
\fill (5,4) circle (0.25);
\fill (5,6) circle (0.25);
\fill (5,7) circle (0.25);
\fill (5,8) circle (0.25);
\fill (5,9) circle (0.25);
\fill (5,10) circle (0.25);
\fill (6,0) circle (0.25);
\fill (6,1) circle (0.25);
\fill (6,2) circle (0.25);
\fill (6,3) circle (0.25);
\fill (6,4) circle (0.25);
\fill (6,5) circle (0.25);
\fill (6,7) circle (0.25);
\fill (6,8) circle (0.25);
\fill (6,9) circle (0.25);
\fill (6,10) circle (0.25);
\fill (7,0) circle (0.25);
\fill (7,1) circle (0.25);
\fill (7,2) circle (0.25);
\fill (7,3) circle (0.25);
\fill (7,4) circle (0.25);
\fill (7,5) circle (0.25);
\fill (7,6) circle (0.25);
\fill (7,8) circle (0.25);
\fill (7,9) circle (0.25);
\fill (7,10) circle (0.25);
\fill (8,0) circle (0.25);
\fill (8,1) circle (0.25);
\fill (8,2) circle (0.25);
\fill (8,3) circle (0.25);
\fill (8,4) circle (0.25);
\fill (8,5) circle (0.25);
\fill (8,6) circle (0.25);
\fill (8,7) circle (0.25);
\fill (8,9) circle (0.25);
\fill (8,10) circle (0.25);
\fill (9,0) circle (0.25);
\fill (9,1) circle (0.25);
\fill (9,2) circle (0.25);
\fill (9,3) circle (0.25);
\fill (9,4) circle (0.25);
\fill (9,5) circle (0.25);
\fill (9,6) circle (0.25);
\fill (9,7) circle (0.25);
\fill (9,8) circle (0.25);
\fill (9,10) circle (0.25);
\fill (10,0) circle (0.25);
\fill (10,1) circle (0.25);

\fill (10,3) circle (0.25);
\fill (10,4) circle (0.25);

\fill (10,6) circle (0.25);
\fill (10,7) circle (0.25);

\fill (10,9) circle (0.25);

\end{scope}
\begin{scope}[xshift=13cm]

\draw[step=1.0,black,thin,xshift=-0.5cm,yshift=-0.5cm] (0,0) grid (11,11);
\fill (0,2) circle (0.25);
\fill (1,2) circle (0.25);
\fill (2,2) circle (0.25);
\fill (0,5) circle (0.25);
\fill (1,5) circle (0.25);
\fill (2,5) circle (0.25);
\fill (0,8) circle (0.25);
\fill (1,8) circle (0.25);
\fill (2,8) circle (0.25);
\end{scope}

\end{tikzpicture}
\caption{The last two layers of the line-free set in Theorem \ref{LBhighp} for $p=11$.}
\end{figure}

\begin{proof}[Proof of Theorem~\ref{LBhighp}]

Let $k=\lfloor \sqrt{p} \rfloor$, $t=\lfloor p/k \rfloor$, $K:=[0,k-1]$ and $T:=\{jk-1\ |\ j\in [1,t]\}$. Consider the set 
\begin{equation*}
\begin{split}   
S:=&[0,p-3]\times [0,p-2]^2 \\ 
\cup &\{p-2\} \times ([0,p-1]^2 \setminus \{(j,j)\ |\ j\in [0,p-1]\}\setminus ((K\cup \{p-1\})\times (T\cup \{p-1\}))   )\\ 
 \cup  & \{p-1\}\times K\times T.
\end{split}
\end{equation*}
we will show that $S\subseteq \mathbb{F}_p^3$ is $p$-progression-free.

Let  $L:=\{(a_1,a_2,a_3)+(b_1,b_2,b_3)i \ |\  i\in [0,p-1]\}$ be a $p$-progression in $\mathbb{F}_p^3$ with $a_1,a_2,a_3,b_1,b_2,b_3\in\mathbb{F}_p$.
\begin{itemize}

\item Case 1: $b_1=0$ and $a_1\neq p-2:$

\noindent
$[0,p-2]^2$ is $p$-progression-free and $|K\times T|=kt<p$, therefore $L$ is not contained in $S$.

\smallskip
\item Case 2: $b_1=0$ and $a_1= p-2:$

\noindent
$L':=\{(a_2,a_3)+(b_2,b_3)i\ |\ i\in [0,p-1]\}$ and $\{(i,i)\ |\ i \in [0,p-1]\}$ are both lines in $\mathbb{F}_p^2$. If they are not parallel or they are equal, they do intersect, and $L$ is not contained in $S$. Otherwise we can rewrite $L'=\{(i,c+i)\ |\ i\in [0,p-1]\}$ with $c\in [1,p-1].$ $\{(i,c+i)\ |\ i\in [0,k-1]\}\cap (K\times (T\cup \{p-1\}))\neq \emptyset$ and therefore $L$ is not contained in $S$.

\smallskip
\item Case 3: $b_1\neq 0$:

\noindent
Without the loss of generality, let $b_1=1$ and $a_1=p-2$. If $b_2=b_3=0$ then $L$ is not contained in $S$ because  the $(p-2)$-layer and $(p-1)$-layer of $S$ have no common point.
Else, if $b_2\neq 0$ and $b_3 \neq 0$ then $\{a_2+b_2i\ |\ i\in [0,p-1]\}=\{a_3+b_3j\ |\ j\in [0,p-1]\}=[0,p-1].$ Since the $(p-2)$-layer is the only layer with $p-1$ entries but $(p-2,p-1,p-1)\not\in S$, $L$ is not contained in $S$.
Finally, if either $b_2=0$ or $b_3=0$ but not both, one of the last two coordinates is constant, and the other one visits every possible value. Now again the  $(p-2)$-layer is the only layer with $p-1$ entries but the $(p-1)$-layer has empty rows and columns wherever the $(p-2)$-layer has $p-1$ entries and therefore $L$ is not contained in $S$.

\end{itemize}
Note that since $p$ is a prime and $k\geq 2$, $t\leq \frac{p-1}{k}$ and that from the definition of $k$ it follows that 
\begin{equation*}
\begin{split}   
&k\in [\sqrt{p}-1,\sqrt{p}+1]\\
\Leftrightarrow \ &k^2-2\sqrt{p}k+p-1\leq 0\\
\Leftrightarrow \ &k+\frac{p-1}{k}\leq 2\sqrt{p},
\end{split}
\end{equation*}
and therefore $k+t\leq 2\sqrt{p}$.
Hence, 
\begin{equation*}
\begin{split}   
|S|=&(p-2)(p-1)^2+(p^2-p-(kt-1)-k-t)+kt \\
=&(p-2)(p-1)^2+p^2-p+1-k-t \\
\geq & (p-2)(p-1)^2+p^2-p+1-2\sqrt{p} \\
=& (p-1)^3 +p-2\sqrt{p}
\end{split}
\end{equation*}

\end{proof}

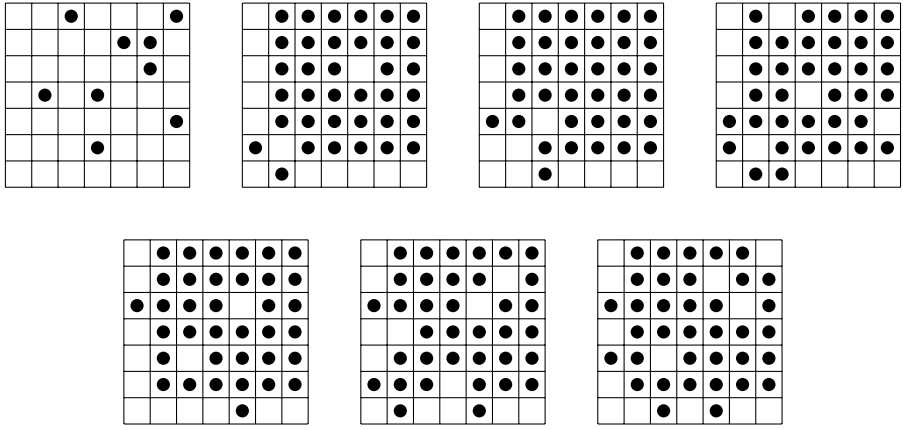
\begin{figure}
\centering
\begin{tikzpicture}[yscale=1,scale=0.35]

\begin{scope}
\draw[step=1.0,black,thin,xshift=-0.5cm,yshift=-0.5cm] (0,0) grid (7,7);

\fill (2,6) circle (0.25);
\fill (6,2) circle (0.25);
\fill (4,5) circle (0.25);
\fill (5,4) circle (0.25);
\fill (1,3) circle (0.25);
\fill (3,1) circle (0.25);
\fill (3,3) circle (0.25);
\fill (5,5) circle (0.25);
\fill (6,6) circle (0.25);

\end{scope}
\begin{scope}[xshift=9cm]

\draw[step=1.0,black,thin,xshift=-0.5cm,yshift=-0.5cm] (0,0) grid (7,7);
\fill (0,1) circle (0.25);
\fill (1,0) circle (0.25);
\fill (1,2) circle (0.25);
\fill (1,3) circle (0.25);
\fill (1,4) circle (0.25);
\fill (1,5) circle (0.25);
\fill (1,6) circle (0.25);
\fill (2,1) circle (0.25);
\fill (2,2) circle (0.25);
\fill (2,3) circle (0.25);
\fill (2,4) circle (0.25);
\fill (2,5) circle (0.25);
\fill (2,6) circle (0.25);
\fill (3,1) circle (0.25);
\fill (3,2) circle (0.25);
\fill (3,3) circle (0.25);
\fill (3,4) circle (0.25);
\fill (3,5) circle (0.25);
\fill (3,6) circle (0.25);
\fill (4,1) circle (0.25);
\fill (4,2) circle (0.25);
\fill (4,3) circle (0.25);
\fill (4,5) circle (0.25);
\fill (4,6) circle (0.25);
\fill (5,1) circle (0.25);
\fill (5,2) circle (0.25);
\fill (5,3) circle (0.25);
\fill (5,4) circle (0.25);
\fill (5,5) circle (0.25);
\fill (5,6) circle (0.25);
\fill (6,1) circle (0.25);
\fill (6,2) circle (0.25);
\fill (6,3) circle (0.25);
\fill (6,4) circle (0.25);
\fill (6,5) circle (0.25);
\fill (6,6) circle (0.25);

\end{scope}
\begin{scope}[xshift=18cm]

\draw[step=1.0,black,thin,xshift=-0.5cm,yshift=-0.5cm] (0,0) grid (7,7);
\fill (0,2) circle (0.25);
\fill (1,2) circle (0.25);
\fill (1,3) circle (0.25);
\fill (1,4) circle (0.25);
\fill (1,5) circle (0.25);
\fill (1,6) circle (0.25);
\fill (2,1) circle (0.25);
\fill (2,0) circle (0.25);
\fill (2,3) circle (0.25);
\fill (2,4) circle (0.25);
\fill (2,5) circle (0.25);
\fill (2,6) circle (0.25);
\fill (3,1) circle (0.25);
\fill (3,2) circle (0.25);
\fill (3,3) circle (0.25);
\fill (3,4) circle (0.25);
\fill (3,5) circle (0.25);
\fill (3,6) circle (0.25);
\fill (4,1) circle (0.25);
\fill (4,2) circle (0.25);
\fill (4,3) circle (0.25);
\fill (4,4) circle (0.25);
\fill (4,5) circle (0.25);
\fill (4,6) circle (0.25);
\fill (5,1) circle (0.25);
\fill (5,2) circle (0.25);
\fill (5,3) circle (0.25);
\fill (5,4) circle (0.25);
\fill (5,5) circle (0.25);
\fill (5,6) circle (0.25);
\fill (6,1) circle (0.25);
\fill (6,2) circle (0.25);
\fill (6,3) circle (0.25);
\fill (6,4) circle (0.25);
\fill (6,5) circle (0.25);
\fill (6,6) circle (0.25);

\end{scope}

\begin{scope}[xshift=27cm]

\draw[step=1.0,black,thin,xshift=-0.5cm,yshift=-0.5cm] (0,0) grid (7,7);
\fill (0,1) circle (0.25);
\fill (0,2) circle (0.25);
\fill (1,0) circle (0.25);
\fill (1,2) circle (0.25);
\fill (1,3) circle (0.25);
\fill (1,4) circle (0.25);
\fill (1,5) circle (0.25);
\fill (1,6) circle (0.25);
\fill (2,0) circle (0.25);
\fill (2,1) circle (0.25);
\fill (2,2) circle (0.25);
\fill (2,3) circle (0.25);
\fill (2,4) circle (0.25);
\fill (2,5) circle (0.25);
\fill (3,1) circle (0.25);
\fill (3,2) circle (0.25);
\fill (3,4) circle (0.25);
\fill (3,5) circle (0.25);
\fill (3,6) circle (0.25);
\fill (4,1) circle (0.25);
\fill (4,2) circle (0.25);
\fill (4,3) circle (0.25);
\fill (4,4) circle (0.25);
\fill (4,5) circle (0.25);
\fill (4,6) circle (0.25);
\fill (5,1) circle (0.25);
\fill (5,2) circle (0.25);
\fill (5,3) circle (0.25);
\fill (5,4) circle (0.25);
\fill (5,5) circle (0.25);
\fill (5,6) circle (0.25);
\fill (6,1) circle (0.25);
\fill (6,3) circle (0.25);
\fill (6,4) circle (0.25);
\fill (6,5) circle (0.25);
\fill (6,6) circle (0.25);

\end{scope}
\begin{scope}[yshift=-9cm, xshift=4.5cm]

\draw[step=1.0,black,thin,xshift=-0.5cm,yshift=-0.5cm] (0,0) grid (7,7);
\fill (0,4) circle (0.25);
\fill (1,1) circle (0.25);
\fill (1,2) circle (0.25);
\fill (1,3) circle (0.25);
\fill (1,4) circle (0.25);
\fill (1,5) circle (0.25);
\fill (1,6) circle (0.25);
\fill (2,1) circle (0.25);
\fill (2,3) circle (0.25);
\fill (2,4) circle (0.25);
\fill (2,5) circle (0.25);
\fill (2,6) circle (0.25);
\fill (3,1) circle (0.25);
\fill (3,2) circle (0.25);
\fill (3,3) circle (0.25);
\fill (3,4) circle (0.25);
\fill (3,5) circle (0.25);
\fill (3,6) circle (0.25);
\fill (4,1) circle (0.25);
\fill (4,2) circle (0.25);
\fill (4,3) circle (0.25);
\fill (4,0) circle (0.25);
\fill (4,5) circle (0.25);
\fill (4,6) circle (0.25);
\fill (5,1) circle (0.25);
\fill (5,2) circle (0.25);
\fill (5,3) circle (0.25);
\fill (5,4) circle (0.25);
\fill (5,5) circle (0.25);
\fill (5,6) circle (0.25);
\fill (6,1) circle (0.25);
\fill (6,2) circle (0.25);
\fill (6,3) circle (0.25);
\fill (6,4) circle (0.25);
\fill (6,5) circle (0.25);
\fill (6,6) circle (0.25);

\end{scope}

\begin{scope}[yshift=-9cm, xshift=13.5cm]

\draw[step=1.0,black,thin,xshift=-0.5cm,yshift=-0.5cm] (0,0) grid (7,7);
\fill (0,1) circle (0.25);
\fill (0,4) circle (0.25);
\fill (1,0) circle (0.25);
\fill (1,1) circle (0.25);
\fill (1,2) circle (0.25);
\fill (1,4) circle (0.25);
\fill (1,5) circle (0.25);
\fill (1,6) circle (0.25);
\fill (2,1) circle (0.25);
\fill (2,2) circle (0.25);
\fill (2,3) circle (0.25);
\fill (2,4) circle (0.25);
\fill (2,5) circle (0.25);
\fill (2,6) circle (0.25);
\fill (4,0) circle (0.25);
\fill (3,2) circle (0.25);
\fill (3,3) circle (0.25);
\fill (3,4) circle (0.25);
\fill (3,5) circle (0.25);
\fill (3,6) circle (0.25);
\fill (4,1) circle (0.25);
\fill (4,2) circle (0.25);
\fill (4,3) circle (0.25);
\fill (4,5) circle (0.25);
\fill (4,6) circle (0.25);
\fill (5,1) circle (0.25);
\fill (5,2) circle (0.25);
\fill (5,3) circle (0.25);
\fill (5,4) circle (0.25);
\fill (5,6) circle (0.25);
\fill (6,1) circle (0.25);
\fill (6,2) circle (0.25);
\fill (6,3) circle (0.25);
\fill (6,4) circle (0.25);
\fill (6,5) circle (0.25);
\fill (6,6) circle (0.25);

\end{scope}

\begin{scope}[yshift=-9cm, xshift=22.5cm]

\draw[step=1.0,black,thin,xshift=-0.5cm,yshift=-0.5cm] (0,0) grid (7,7);
\fill (0,2) circle (0.25);
\fill (0,4) circle (0.25);
\fill (1,1) circle (0.25);
\fill (1,2) circle (0.25);
\fill (1,3) circle (0.25);
\fill (1,4) circle (0.25);
\fill (1,5) circle (0.25);
\fill (1,6) circle (0.25);
\fill (2,0) circle (0.25);
\fill (2,1) circle (0.25);
\fill (2,3) circle (0.25);
\fill (2,4) circle (0.25);
\fill (2,5) circle (0.25);
\fill (2,6) circle (0.25);
\fill (3,1) circle (0.25);
\fill (3,2) circle (0.25);
\fill (3,3) circle (0.25);
\fill (3,4) circle (0.25);
\fill (3,5) circle (0.25);
\fill (3,6) circle (0.25);
\fill (4,0) circle (0.25);
\fill (4,1) circle (0.25);
\fill (4,2) circle (0.25);
\fill (4,3) circle (0.25);
\fill (4,4) circle (0.25);
\fill (4,6) circle (0.25);
\fill (5,1) circle (0.25);
\fill (5,2) circle (0.25);
\fill (5,3) circle (0.25);
\fill (5,5) circle (0.25);
\fill (5,6) circle (0.25);
\fill (6,1) circle (0.25);
\fill (6,2) circle (0.25);
\fill (6,3) circle (0.25);
\fill (6,4) circle (0.25);
\fill (6,5) circle (0.25);

\end{scope}

\end{tikzpicture}
\caption{The line-free set in Theorem \ref{LB7} for $p=7$.}
\end{figure}

\begin{proof}[Proof of Theorem \ref{LB7}]
Let $p=7+24\ell$ for $\ell\in\mathbb{Z}_{\geq0}$, let $A$ be the set of quadratic residues, that is, $A=\{a^2\ |\ a\in\mathbb{F}_p^*\}$ and $B:=\mathbb{F}_p^*\setminus A$. Note that $|A|=|B|=\frac{p-1}{2}$ and the law of quadratic reciprocity yields 

$$\legendre{-1}{p}=(-1)^{\frac{p-1}{2}}=(-1)^{3+12\ell}=-1,$$

$$\legendre{2}{p}=(-1)^{\frac{p^2-1}{8}}=(-1)^{6+42\ell+72\ell^2}=1,$$

$$\legendre{3}{p}=(-1)^{\frac{p-1}{2}\frac{3-1}{2}}\legendre{p}{3}=(-1)^{3+12\ell}\legendre{1}{3}=-1,$$

 and therefore $2\in A$ and $\{-1,3\}\subseteq B$. Note here that $A$ is a subgroup of $\mathbb{F}_p^*$ and this means that multiplication by $2$ or $\frac{1}{2}$ leaves elements of $A$ or $B$ in the same set, while multiplication by $-1$, $3$ or $\frac{1}{3}$ changes the set. For instance, $3a\in B$ for all $a\in A$ and $-\frac{3b}{2}=(-1)\cdot 3\cdot \frac{1}{2} \cdot b \in B$ for all $b\in B$.
Let 
\begin{equation*}
\begin{split}   
S:=&[1,p-1]^3\cup \big(\{(a,0,a)|a\in A\}\cup \{(0,a,a)|a\in A\}\big) \\ 
\setminus &\big(\{(a,a,a)|a\in A\}\cup \{(a/2,a/2,a)|a\in A\}\big) \\ 
 \cup &\big(\{(3b/2,0,b)|b\in B\}\cup \{(0,3b/2,b)|b\in B\}\cup (3b,0,b)|b\in B\}\cup \{(0,3b,b)|b\in B\}\big) \\
\setminus &\big(\{(b,b,b)|b\in B\}\cup \{(3b/2,3b/2,b)|b\in B\}\cup\{(b/3,b/3,b)|b\in B\}\big) \\
\setminus &\big( \{(3b,-3b/2,b)|b\in B\}\cup \{(-3b/2,3b,b)|b\in B\}\big) \\
\cup &\big(\{(b,b,0)|b\in B\}\cup \{(2a,-a,0)|a\in A\}\cup (-a,2a,0)|a\in A\}\big).
\end{split}
\end{equation*}
We will show that $S$ is $p$-progression-free.

Note that $S$ is symmetric in the first two coordinates. We will therefore, in this proof, skip one of two symmetric cases, whenever possible.

Let  $L:=\{(c_1,c_2,c_3)+(d_1,d_2,d_3)i \ |\  i\in [0,p-1]\}$ be a $p$-progression in $\mathbb{F}_p^3$ with $c_1,c_2,c_3,d_1,d_2,d_3\in\mathbb{F}_p$ and assume that $L\subseteq S$.

First, assume that $d_3=0$. 
\begin{itemize}
\item Case 1: $c_3=0$:

\noindent
Since $S$ contains no points where the third and one of the first two coordinates is $0$, $L$ is not contained in $S$.

\smallskip
\item Case 2: $c_3\in A$:

\noindent
Let $a:=c_3$. Since $(a, 0, a)$  and $(0,a,a)$ are the only points where the third coordinate is $a$ and one of the first two coordinates is $0$, we can assume $(a, 0, a)\in L$. If $d_1=0$ then $(a,a,a)\in L$, a contradiction. If $d_1\neq 0$ then also $(0,a,a)\in L$ and consequently $(\frac{a}{2},\frac{a}{2},a)=\frac{1}{2}(a, 0, a)+\frac{1}{2}(0,a,a)\in L$, again a contradiction.

\smallskip
\item Case 3: $c_3\in B$: 

\noindent
Let $b:=c_3$. First, assume $d_1\neq0$ and $d_2\neq0$. Since  $(\frac{3b}{2}, 0, b)$, $(0, \frac{3b}{2}, b)$, $(3b, 0, b)$  and $(0,3b,b)$ are the only points where the third coordinate is $b$ and one of the first two coordinates is $0$, we only have to consider the following cases:

If $(\frac{3b}{2}, 0, b)\in L$ and $(0, \frac{3b}{2}, b)\in L$, then also $(-\frac{3b}{2},3b , b)=(-1)(\frac{3b}{2}, 0, b)+2(0, \frac{3b}{2}, b)\in L$, if  $(\frac{3b}{2}, 0, b)\in L$ and $(0,3b,b)\in L$, then also $(b,b , b)=\frac{2}{3}(\frac{3b}{2}, 0, b)+\frac{1}{3}(0, 3b, b)\in L$ and if $(3b, 0, b)\in L$ and $(0,3b,b)\in L$, then also $(\frac{3b}{2},\frac{3b}{2}, b)=\frac{1}{2}(3b, 0, b)+\frac{1}{2}(0, 3b, b)\in L$. Consequently, we arrived at a contradiction.
Now, if $d_1=0$ or $d_2=0$, again $L$ has to contain one of the points $(\frac{3b}{2}, 0, b)$, $(0, \frac{3b}{2}, b)$, $(3b, 0, b)$, $(0,3b,b)$ and therefore $L$ also contains one of the points $(\frac{3b}{2},\frac{3b}{2},b)$, $(3b,\frac{-3b}{2},b)$, $(\frac{-3b}{2},3b,b)$, again a contradiction.

\end{itemize}

Now assume that $d_3\neq0$. If $d_1=d_2=0$, $L$ contains a point with a zero last coordinate. We get that either $(b,b,0)$ and therefore also $(b,b,b)$ is in $L$ for some $b\in B$ or $(2a,-a,0)\in L$ for some $a\in A$ and therefore also $(3b,\frac{-3b}{2},b)\in L$ for the unique $b\in B$ such that $3b=2a$, both a contradiction.

In the remaining case $d_3\neq 0$ and at least one of $d_1$ and $d_2$ is non-zero. Since there is no point in $S$ where both the third and one of first two coordinates is zero, $L$ has to include a point with third coordinate being zero and a different point where one of the other two coordinates is zero. We are therefore left with checking the following cases, where $L$ is given by a pair of two points in $S$. For some of these cases it is important to note that $S$ contains no points where one of the first two coordinates is $0$ and the other is in $B$.

\begin{itemize}
    \item $(b,b,0)\in L$ and $(0,a,a)\in L$: $$L=\bigg\{\begin{pmatrix}
b\\
b\\
0
\end{pmatrix}+k\begin{pmatrix}
-b\\
a-b\\
a
\end{pmatrix}\bigg|\ k\in\mathbb{F}_p\bigg\}$$

Since $a\neq b$, setting $k:=\frac{b}{b-a}$, we get $(x,y,z):=(-\frac{ab}{b-a},0,\frac{ab}{b-a})\in L$. Now $x=-z$ and therefore $z\in B$ and $x\in A$, so $-1=\frac{3}{2}$ or $-1=3$, a contradiction.

    \item $(b,b,0)\in L$ and $(0,\frac{3b'}{2},b')\in L$: $$L=\bigg\{\begin{pmatrix}
b\\
b\\
0
\end{pmatrix}+k\begin{pmatrix}
-b\\
\frac{3b'}{2}-b\\
b'
\end{pmatrix}\bigg|\ k\in\mathbb{F}_p\bigg\}$$

Since $\frac{3b'}{2}\neq b$, setting $k:=\frac{b}{b-\frac{3b'}{2}}$, we get $(x,y,z):=(-\frac{3bb'}{2(b-\frac{3b'}{2})},0,\frac{bb'}{b-\frac{3b'}{2}})\in L$. Now $x=-\frac{3}{2}z$ and therefore $x,z\in A$, so $-\frac{3}{2}=1$, a contradiction.

    \item $(b,b,0)\in L$ and $(0,3b',b')\in L$: $$L=\bigg\{\begin{pmatrix}
b\\
b\\
0
\end{pmatrix}+k\begin{pmatrix}
-b\\
3b'-b\\
b'
\end{pmatrix}\bigg|\ k\in\mathbb{F}_p\bigg\}$$

Since $3b'\neq b$, setting $k:=\frac{b}{b-3b'}$, we get $(x,y,z):=(-\frac{3bb'}{b-3b'},0,\frac{bb'}{b-3b'})\in L$. Now $x=-3z$ and therefore $x,z\in A$, so $-3=1$, a contradiction.

    \item $(2a,-a,0)\in L$ and $(0,a',a')\in L$: $$L=\bigg\{\begin{pmatrix}
2a\\
-a\\
0
\end{pmatrix}+k\begin{pmatrix}
-2a\\
a'+a\\
a'
\end{pmatrix}\bigg|\ k\in\mathbb{F}_p\bigg\}$$

Since $a'\neq -a$, setting $k:=\frac{a}{a+a'}$, we get $(x,y,z):=(\frac{2aa'}{a+a'},0,\frac{aa'}{a+a'})\in L$. Now $x=2z$ and therefore $x,z\in A$, so $2=1$, a contradiction.

    \item $(2a,-a,0)\in L$ and $(a',0,a')\in L$: $$L=\bigg\{\begin{pmatrix}
2a\\
-a\\
0
\end{pmatrix}+k\begin{pmatrix}
a'-2a\\
a\\
a'
\end{pmatrix}\bigg|\ k\in\mathbb{F}_p\bigg\}$$

Assume $a'\neq 2a$, then setting $k:=\frac{2a}{2a-a'}$, we get $(x,y,z):=(0,\frac{aa'}{2a-a'},\frac{2aa'}{2a-a'})\in L$. Now $2y=z$ and therefore $y,z\in A$, so $1=2$, a contradiction.
If $a'=2a$, then setting $k:=3$, we get $(x,y,z):=(2a,2a,6a)\in L$, a contradiction since $6a\in B$.

\item $(2a,-a,0)\in L$ and $(\frac{3b}{2},0,b)\in L$: $$L=\bigg\{\begin{pmatrix}
2a\\
-a\\
0
\end{pmatrix}+k\begin{pmatrix}
\frac{3b}{2}-2a\\
a\\
b
\end{pmatrix}\bigg|\ k\in\mathbb{F}_p\bigg\}$$

Assume $\frac{3b}{2}\neq 2a$, then setting $k:=\frac{2a}{2a-\frac{3b}{2}}$, we get $(x,y,z):=(0,\frac{3ab}{2(2a-\frac{3b}{2})},\frac{2ab}{2a-\frac{3b}{2}})\in L$. Now $y=\frac{3z}{4}$ and therefore $z\in B$ and $x\in A$, so $\frac{3}{4}=\frac{3}{2}$ or $\frac{3}{4}=3$, a contradiction.
If $\frac{3b}{2}= 2a$, then setting $k:=3$, we get $(x,y,z):=(2a,2a,4a)\in L$, a contradiction since $4a\in A$.

\item $(2a,-a,0)\in L$ and $(0,\frac{3b}{2},b)\in L$: $$L=\bigg\{\begin{pmatrix}
2a\\
-a\\
0
\end{pmatrix}+k\begin{pmatrix}
-2a\\
\frac{3b}{2}+a\\
b
\end{pmatrix}\bigg|\ k\in\mathbb{F}_p\bigg\}$$

Assume $\frac{3b}{2}\neq -3a$, then setting $k:=\frac{3a}{\frac{3b}{2}+3a}$, we get $(x,y,z):=(\frac{3ab}{\frac{3b}{2}+3a},\frac{3ab}{\frac{3b}{2}+3a},\frac{3ab}{\frac{3b}{2}+3a})\in L$. Now $x=y=z$, a contradiction.
If $\frac{3b}{2}= -3a$, then setting $k:=-\frac{1}{2}$, we get $(x,y,z):=(3a,0,a)\in L$, so $1=3$, a contradiction.

\item $(2a,-a,0)\in L$ and $(3b,0,b)\in L$: $$L=\bigg\{\begin{pmatrix}
2a\\
-a\\
0
\end{pmatrix}+k\begin{pmatrix}
3b-2a\\
a\\
b
\end{pmatrix}\bigg|\ k\in\mathbb{F}_p\bigg\}$$

Since $b\neq a$, then setting $k:=\frac{a}{a-b}$, we get $(x,y,z):=(\frac{ab}{a-b},\frac{ab}{a-b},\frac{ab}{a-b})\in L$. Now $x=y=z$, a contradiction.

\item $(2a,-a,0)\in L$ and $(0,3b,b)\in L$: $$L=\bigg\{\begin{pmatrix}
2a\\
-a\\
0
\end{pmatrix}+k\begin{pmatrix}
-2a\\
a+3b\\
b
\end{pmatrix}\bigg|\ k\in\mathbb{F}_p\bigg\}$$

since $3b\neq -a$, setting $k:=\frac{a}{a+3b}$, we get $(x,y,z):=(\frac{6ab}{a+3b},0,\frac{ab}{a+3b})\in L$. Now $x=6z$ and therefore $z\in B$ and $x\in A$, so $6=\frac{3}{2}$ or $6=3$, a contradiction.

\end{itemize}

Finally, note that the layer with third coordinate $0$ contains $|B|+2|A|=\frac{3}{2}(p-1)$ points, layers with the third coordinate in $A$ contain $(p-1)^2$ points and layers with the third coordinate in $B$ contain $(p-1)^2-1$ points and thus $$|S|=\frac{p-1}{2}(p-1)^2+\frac{p-1}{2}((p-1)^2-1)+\frac{3}{2}(p-1)=(p-1)^3+(p-1).$$

In the special case of $p=7$, the layers with third coordinate in $B$ actually contain $(p-1)^2=36$ points, since $\frac{3}{2}=\frac{1}{3}$, thus giving the lower bound of $225$.

\end{proof}

\subsection*{Acknowledgements}
C.E. was supported by a joint 
 FWF-ANR project ArithRand
 (I 4945-N and ANR-20-CE91-0006).
J.F. was supported by the Austrian Science Fund (FWF) under the project W1230.
D.G.S. was supported by the ERC Advanced Grant "Geoscape". 
B.K. was supported by the ÚNKP-23-3, New National Excellence Program of the Ministry for Culture and Innovation from the source of the National Research, Development and Innovation Fund. P.P.P. was supported by the Lend\"ulet program of the Hungarian Academy of Sciences (MTA) and by the National Research, Development and Innovation Office NKFIH (Grant Nr. K129335 and K146387). E.F., B.K., D.G.S. and N.V. would like to thank the Hungarian REU 2022 program. The authors also thank Zolt\'an L\'or\'ant Nagy for comments on this manuscript.

\printbibliography 

\clearpage
\section{Appendix}

\begin{figure}[H]
\centering
\begin{tikzpicture}[yscale=1,scale=0.5]

\begin{scope}
\draw[step=1.0,black,thin,xshift=-0.5cm,yshift=-0.5cm] (0,0) grid (5,5);

\fill (3,1) circle (0.25);
\fill (1,3) circle (0.25);
\fill (3,2) circle (0.25);
\fill (2,3) circle (0.25);
\fill (2,4) circle (0.25);
\fill (4,2) circle (0.25);

\end{scope}
\begin{scope}[xshift=7cm]

\draw[step=1.0,black,thin,xshift=-0.5cm,yshift=-0.5cm] (0,0) grid (5,5);
\fill (0,1) circle (0.25);
\fill (1,0) circle (0.25);
\fill (1,2) circle (0.25);
\fill (1,3) circle (0.25);
\fill (1,4) circle (0.25);
\fill (2,1) circle (0.25);
\fill (2,2) circle (0.25);
\fill (2,3) circle (0.25);
\fill (2,4) circle (0.25);
\fill (3,1) circle (0.25);
\fill (3,2) circle (0.25);
\fill (3,4) circle (0.25);
\fill (4,1) circle (0.25);
\fill (4,2) circle (0.25);
\fill (4,3) circle (0.25);
\fill (4,4) circle (0.25);

\end{scope}
\begin{scope}[xshift=14cm]

\draw[step=1.0,black,thin,xshift=-0.5cm,yshift=-0.5cm] (0,0) grid (5,5);
\fill (0,3) circle (0.25);
\fill (0,4) circle (0.25);
\fill (1,1) circle (0.25);
\fill (1,2) circle (0.25);
\fill (1,4) circle (0.25);
\fill (2,1) circle (0.25);
\fill (2,2) circle (0.25);
\fill (2,3) circle (0.25);
\fill (2,4) circle (0.25);
\fill (3,0) circle (0.25);
\fill (3,1) circle (0.25);
\fill (3,3) circle (0.25);
\fill (3,4) circle (0.25);
\fill (4,0) circle (0.25);
\fill (4,1) circle (0.25);
\fill (4,3) circle (0.25);

\end{scope}

\begin{scope}[yshift=-7cm, xshift=3.5cm]

\draw[step=1.0,black,thin,xshift=-0.5cm,yshift=-0.5cm] (0,0) grid (5,5);
\fill (0,3) circle (0.25);
\fill (0,4) circle (0.25);
\fill (1,1) circle (0.25);
\fill (1,2) circle (0.25);
\fill (1,4) circle (0.25);
\fill (2,1) circle (0.25);
\fill (2,2) circle (0.25);
\fill (3,2) circle (0.25);
\fill (4,2) circle (0.25);
\fill (3,0) circle (0.25);
\fill (1,3) circle (0.25);
\fill (3,3) circle (0.25);
\fill (3,4) circle (0.25);
\fill (4,0) circle (0.25);
\fill (4,1) circle (0.25);
\fill (4,3) circle (0.25);

\end{scope}
\begin{scope}[yshift=-7cm, xshift=10.5cm]

\draw[step=1.0,black,thin,xshift=-0.5cm,yshift=-0.5cm] (0,0) grid (5,5);
\fill (0,1) circle (0.25);
\fill (1,0) circle (0.25);
\fill (1,2) circle (0.25);
\fill (1,3) circle (0.25);
\fill (1,4) circle (0.25);
\fill (2,1) circle (0.25);
\fill (2,2) circle (0.25);
\fill (2,3) circle (0.25);
\fill (2,4) circle (0.25);
\fill (3,1) circle (0.25);
\fill (3,2) circle (0.25);
\fill (3,4) circle (0.25);
\fill (4,1) circle (0.25);
\fill (4,2) circle (0.25);
\fill (4,3) circle (0.25);
\fill (4,4) circle (0.25);

\end{scope}

\end{tikzpicture}
\caption{A line-free set showing $r_p(\mathbb{F}_p^3)\geq 70$}
\label{fig70}
\end{figure}
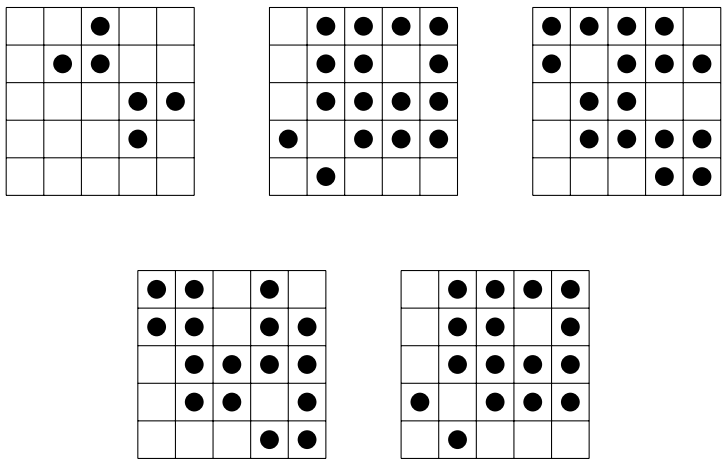

\begin{table}[h!]
\centering
\begin{tabular}{ |c|c|c|c|c|c| } 
\hline
$p$ & $5$ & $7$ & $11$ & $13$ & $17$ \\
\hline
$n=3$ & $4.041$ & $6.027$ & $10.016$ & $12.013$ & $16.010$ \\ 
$n=4$ & $4.046$ & $6.034$ & $10.022$ & $12.019$ & $16.014$ \\ 
$n=5$ & \textcolor{lightgray}{$4.041$} & \textcolor{lightgray}{$6.034$}  & $10.024$ & $12.020$ & $16.016$ \\ 
$n=6$ & \textcolor{lightgray}{$4.034$}  & \textcolor{lightgray}{$6.031$}  & \textcolor{lightgray}{$10.024$}   & $12.021$ & $16.017$ \\ 
$n=7$ & \textcolor{lightgray}{$4.027$}  & \textcolor{lightgray}{$6.028$}  & \textcolor{lightgray}{$10.023$}  & \textcolor{lightgray}{$12.020$}  & $16.017$ \\ 
\hline
$n=2p$ & $4.090$ & $6.066$ & $10.043$ & $12.037$ & $16.028$ \\ 
\hline

\end{tabular}
\caption{The bases $a$ for the lower bounds $r_p(\mathbb{F}_p^n)\geq a^n$ that can be achieved by Theorem~\ref{LBhighn} for small primes. The last row gives bounds that can only be used for dimensions at least $2p$, using the results of Frankl et al.~\cite{FGR}.}
\label{table:1}
\end{table}

\end{document}